\documentclass[10pt,reqno]{amsart}
\usepackage{amsmath, amsthm, amscd, amsfonts, amssymb, graphicx, color, float}
\usepackage[bookmarksnumbered, colorlinks, plainpages]{hyperref}
\usepackage[all]{xy}
\hypersetup{colorlinks=true,linkcolor=red, anchorcolor=green, citecolor=cyan, urlcolor=red, filecolor=magenta, pdftoolbar=true}

\newtheorem{theorem}{Theorem}[section]

\newtheorem{proposition}[theorem]{Proposition}
\newtheorem{corollary}[theorem]{Corollary}
\theoremstyle{definition}
\newtheorem{definition}[theorem]{Definition}
\newtheorem{example}[theorem]{Example}

\theoremstyle{remark}

\numberwithin{equation}{section}

\begin{document}

\title{Inductive limits of compact quantum metric spaces}

\author[Botao Long]{Botao Long$^{*,1}$}
\address{$^1$ School of Mathematics, MIIT Key Laboratory of Mathematical Modelling
and High Performance Computing of Air Vehicles, Nanjing University of Aeronautics and Astronautics, Nanjing 211106, Jiangsu, China}
\email{longbt289@nuaa.edu.cn; longbt289@163.com}

\author[Ghadir Sadeghi]{Ghadir Sadeghi $^{2}$}

\address{$^2$ Department of Mathematics and Computer Sciences, Hakim Sabzevari University, P. O. Box 397, Sabzevar, Iran}
\email{g.sadeghi@hsu.ac.ir; ghadir54@gmail.com}

\renewcommand{\subjclassname}{\textup{2020} Mathematics Subject Classification}

\thanks{$^*$Corresponding author: Long Botao}



\begin{abstract}
A compact quantum metric space is a unital $C^*$-algebra equipped with a Lip-norm.
Let $\{(A_n, L_n)\}$ be a sequence of compact quantum metric spaces, and let $\phi_n:A_n\to A_{n+1}$ be a unital $^*$-homomorphism preserving Lipschitz elements for $n\geq 1$.
We show that there exists a compact quantum metric space structure on the inductive limit $\varinjlim(A_n,\phi_n)$ by means of the inverse limit of the state spaces $\{\mathcal{S}(A_n)\}$. 
We also give some sufficient conditions that two inductive limits of compact quantum metric spaces are Lipschitz isomorphic.
\end{abstract}

\maketitle



\section{Introduction}

In 1989, Connes initiated the study of noncommutative metric spaces in terms of a spectral triple, by which he formulated the metric data in noncommutative geometry \cite{Connes1989,Connes1994}.
He gave precisely the geodesic distance on a compact, spin and Riemannian manifold $M$ back by a spectral triple $(C(M),L^2(M,S),D)$, which consists of the unital commutative $C^*$-algebra $C(M)$ of complex-valued continuous functions on $M$ represented faithfully by pointwise multiplication operators on the Hilbert space $L^2(M,S)$ of $L^2$ spinors on $M$, and of a Dirac operator $D$ on $L^2(M,S)$.
 More specifically, the geodesic distance $\rho(p,q)$ between any two points $p,q$ of $M$ is computed via the Dirac operator $D$ through the following formula:
\begin{equation*}\label{geodesic1}
  \rho(p,q)=\sup\{|f(p)-f(q)|:f\in C(M),\|[D,f]\|\leq 1\}.
\end{equation*}
For the compact metric space $M$, we can think of the points $p,q$ in $M$ to be characters (or pure states) $\hat{p},\hat{q}$ of $C(M)$, and thus the above formula can be regarded as follows:
\begin{equation*}\label{geodesic2}
  \rho(p,q)=\sup\{|\hat{p}(f)-\hat{q}(f)|:f\in C(M),\|[D,f]\|\leq 1\}.
\end{equation*}
Inspired by this key observation, Connes extended the concept of metric to the noncommutative setting
by way of introducing an ordinary metric $\rho_{L_D}$ on the state space $\mathcal{S}(A)$ of a unital noncommutative $C^*$-algebra $A$ from a spectral triple $(A,\mathcal{H},D)$ by a similar formula
\begin{equation}\label{spectral metric}
  \rho_{L_D}(\mu,\nu)=\sup\{|\mu(a)-\nu(a)|:a\in A,L_D(a)=\|[D,a]\|\leq 1\},
\end{equation}
for $\mu,\nu\in\mathcal{S}(A)$.
This metric is a generalization of the Monge-Kantorovich metric on the set of all regular probability measures on a compact Hausdorff space $X$, which is identified with the state space of $C(X)$ by Riesz's representation theorem \cite{R2,R5}.
 As a consequence, a spectral triple is  appropriate for a noncommutative analog of a metric on a compact space.

 In \cite{CI2}, Christensen and Ivan obtained two classes of spectral triples with respect to $C(X)$ of an ordinary compact metric space $(X,\rho)$, applying countable direct sums of two-dimensional modules.
 The first class of spectral triple is finitely summable for any positive real number, and recovers the original metric $\rho$ by the formula \eqref{spectral metric} precisely.
 The second class of spectral triple is parameterized by a real number $\delta>0$, and does not give the original metric $\rho$ back exactly. Its induced metric $\rho_\delta$ on $X$ is only within a $\delta$-distance of $\rho$, that is,
\[
\rho(p,q)\leq \rho_\delta(p,q)\leq (1+\delta)\rho(p,q)
\]
for all $p,q\in X$.
In addition, the second class of spectral triple reflects some aspects of the topological dimensions of the compact metric space $(X,\rho)$, and gives some computable estimates of the upper Minkowski dimension of the metric space $(X,\rho)$.

In 1998, stimulated by what happens for ordinary compact metric spaces, Rieffel initiated the discussion of the agreement between the underlying weak $^*$-topology on the state space $\mathcal{S}(A)$ and the metric topology determined by the metric $\rho_{L_D}$ arising from the above formula \eqref{spectral metric}.
The metric data for a unital noncommutative $C^*$-algebra $A$ was presented through the seminorm $L_D(a):=\|[D,a]\|$, which acts as the usual Lipschitz seminorms for ordinary compact metric spaces \cite{R5}.
Christensen and Ivan, and Hawkins and Zacharias produced this kind of examples on $C^*$-algebraic extensions of unital $C^*$-algebras by stable ideals under a certain Toeplitz type property \cite{CI3,HZ}.
Recently, it turns out that   several classes of crossed product $C^*$-algebras satisfy the agreement between the metric topology and the weak $^*$-topology \cite{Haw,JP,KK,Long5,Long7}. Many other exciting examples of this situation have been constructed as well \cite{CAEC,CR,CI2,Connes1989,Li2005,Long1,Long3,OR,R1,R3}.

In general, if there is a $^*$-seminorm $L$ on a dense $^*$-subalgebra of a unital $C^*$-algebra $A$ with the identity element $\mathbf{1}_A$ such that $L(\mathbf{1}_A)=0$, we then obtain a metric $\rho_L$ on the state space $\mathcal{S}(A)$ of $A$, much as Connes did, by
\[\rho_L(\mu,\nu)=\sup\{|\mu(a)-\nu(a)|:a\in A, L(a)\leq1\},\quad \mu,\nu\in\mathcal{S}(A).\]
(Without further hypotheses $\rho_L$ may take the value $+\infty$.)
When the induced metric topology on $\mathcal{S}(A)$ arising from $\rho_L$ coincides with the underlying weak $^*$-topology, Rieffel defined the pair $(A,L)$ to be a compact quantum metric space \cite{R4,R5}.
By introducing a notion of Gromov-Hausdorff distance for compact quantum metric spaces, Rieffel can give a precise meaning to the statement that a sequence of matrix algebras converge (in this quantum distance) to the $2$-sphere \cite{R4, R6}, which appears in the literature of theoretical high-energy physics and string theory (see \cite{R4, R6} and references therein).
See \cite{KD2,La2015,La2016b,Li2003,Li2006,R7,Wu2005,Wu2006p,Wu2006f} for further discussion.

In \cite{R1}, Rieffel showed that if there is an ergodic action $\alpha$ of a compact group $G$ with the identity element $e$ and a continuous length function $\ell$ on a unital $C^*$-algebra $A$, then the seminorm
\begin{equation}\label{ergodic}
  L(a)=\sup\left\{ \frac{\|\alpha_g(a)-a\|}{\ell(g)}:g\neq e\right \},\quad a\in A
\end{equation}
endows $A$ with a compact quantum metric space structure.
Let $M_p$ be the $p\times p$ matrix algebra over $\mathbb{C}$. There is a unique ergodic action of $\mathbb{Z}_p\times\mathbb{Z}_p$ on $M_p$ up to conjugacy,
which can induce an ergodic action of $(\mathbb{Z}_p\times\mathbb{Z}_p)^{\mathbb{Z}}$ on the UHF algebra $M_{p^{\infty}}$ \cite{KD1}.
 Moreover, Kerr introduced a continuous length function on $(\mathbb{Z}_p\times\mathbb{Z}_p)^{\mathbb{Z}}$, and computed the metric dimension of $M_{p^{\infty}}$ with respect to the seminorm arising from this length function and the formula \eqref{ergodic} \cite{KD1}.
Furthermore, it is shown that if $A$ is a UHF algebra with an $^*$-automorphism $\alpha$ fixing $a$ UHF-filtration, then $A\rtimes_\alpha \mathbb{Z}$ is a compact quantum metric space \cite{HZ}.

The typical example of compact quantum metric space is noncommutative $N$-tori \cite{Long1,Long2,R1,R3}.
Let $A_\theta$ be the noncommutative $2$-torus with generators $U$ and $V$ satisfying the relation $VU=e^{2\pi i\theta}UV$ for some irrational number $\theta$. There is a $^*$-automorphism $\beta$ of $A_\theta$ by $\beta(U)=U$ and $\beta(V)=V^{-1}$. Bratteli and Kishimoto verif\/ied that the f\/ixed point algebra $B_\theta$ of $\beta$ is actually an AF algebra \cite{BK}. Rieffel pointed out in \cite{R5} that $B_\theta$ is also a compact quantum metric space. For a unital AF algebra $A$ with a faithful state, there is a natural filtration, as an increasing sequence of finite dimensional $C^*$-algebras, on $A$, by which Christensen and Ivan constructed a ($p$-summable) spectral triple, and hence induced a compact quantum metric space structure on $A$ \cite{CI1}. In \cite{AL1}, Aguilar and Latr\'{e}moli\`{e}re constructed compact quantum metric space structures on unital AF algebras with a faithful tracial state, and proved that for such metrics, AF algebras are the limits of their defining inductive sequences of finite dimensional $C^*$-algebras for the quantum propinquity.

In \cite{ALR}, Aguilar, Latr\'{e}moli\`{e}re and Rainone showed that Bunce-Deddens algebras, as compact quantum metric spaces, are also limits of circle algebras for Rieffel's quantum Gromov-Hausdorff distance, and form a continuous family indexed by the Baire space.
Given a unital inductive limit of $C^*$-algebras for which each $C^*$-algebra of the inductive sequence is endowed with a compact quantum metric space structure, Aguilar produced sufficient conditions to build a compact quantum metric on the inductive limit from the quantum metrics on the inductive sequence by utilizing the completeness of the dual Gromov-Hausdorff propinquity \cite{A1}. Therefore, it is a natural question whether one may use some more relaxed conditions to give a compact quantum metric space structure on the unital inductive limit with more general building blocks such that more interested $C^*$-algebraic classes in operator algebras can be equipped with a compact quantum metric space structure.

In the present paper, we propose a solution to the above question.
The contents of the sections of this paper are as follows.
   In Section 2, we provide some basic concepts of the theory of compact quantum metric spaces.
   In Section 3, we introduce a notion of inductive sequence of compact quantum metric spaces with connecting unital $^*$-homomorphisms preserving Lipschitz elements, and endow the resulting inductive limit with a compact quantum metric space structure by means of the inverse limit of the state spaces of building blocks.
   In Section 4, we give some sufficient conditions that two inductive limits of compact quantum metric spaces are Lipschitz isomorphic by way of the inductive limits and their building blocks.


\section{Preliminaries}

In this section, we provide some preliminaries and background for the theory of compact quantum metric spaces. We start with ordinary compact metric space since it is a motivation for the definition of a compact quantum metric space.

To this end, let $(X,d)$ be a compact metric space, and let $C(X)$ be the $C^*$-algebra of all complex-valued  continuous functions on $X$. For any $f\in C(X)$, we define the Lipschitz constant of $f$ as
\begin{eqnarray*}
L_d(f)=\sup\left\{\frac{|f(x)-f(y)|}{d(x,y)}: x,y\in X, x\neq y\right\},
\end{eqnarray*}
where the value $+\infty$ is permitted. Note that the metric $d$  can be exactly recovered from $L_d$ by the formula
\[
d(x,y)=\sup\left\{|f(x)-f(y)|:f\in C(X), L_d(f)\leq1\right\},
\]
for all $x,y\in X$. One constructs a metric, which is called the Monge-Kantorovich metric, on the state space $\mathcal{S}(C(X))$ of $C(X)$, i.e., the set of all probability measures on $X$, by
\[
\rho_{L_d}(\mu,\nu)=\sup\left\{|\mu(f)-\nu(f)|:f\in C(X), L_d(f)\leq1\right\},
\]
for all $\mu,\nu\in\mathcal{S}(C(X))$.
The metric $\rho_{L_d}$ extends $d$ from the set of all Dirac measures on $X$ to the set of all probability measures.
Kantorovich showed that the topology on the state space $\mathcal{S}(C(X))$ induced by the metric $\rho_{L_d}$ coincides with the underlying weak$^*$-topology \cite{R1,R5}.

Let $ A$ be a unital $C^*$-algebra. The identity element of $ A$ is denoted by $\mathbf{1}_{ A}$.
The state space of $ A$ is represented by $\mathcal{S}( A)$, and the self-adjoint part of $ A$  is denoted
by $ A_{sa}$.

 A \textit{Lipschitz seminorm} on a unital $C^*$-algebra $A$ is a seminorm $L$ on $A$ that is permitted to take the value $+\infty$, and satisfies
  \begin{enumerate}
    \item $L(a^*)=L(a)$ for all $a\in A$.
    \item $L(a)=0$ if and only if $a\in\mathbb{C}\mathbf{1}_A$.
    \item The set $\mathrm{dom}(L)=\{a\in A:L(a)<+\infty\}$ of Lipschitz elements in $A$ is a dense subspace of $A$.
  \end{enumerate}
If the set $\{a\in A:L(a)\leq r\}$ is closed in $A$ for some and hence all $r>0$, we say that $L$ is \textit{lower semicontinuous}. Equivalently, for any sequence $\{a_n\}$ in $A$ which converges in norm to $a\in A$, we have $L(a)\leq\liminf_{n\to\infty}L(a_n)$.

  A \textit{Lip-norm} on a unital $C^*$-algebra $A$ is a Lipschitz seminorm $L$ such that the topology, induced by the Monge-Kantorovich metric
  \[
 \rho_L(\mu,\nu)=\sup\{|\mu(a)-\nu(a)|: a\in A, L(a)\leq1\},\quad \mu, \nu\in\mathcal{S}(A)
 \]
 on the state space $\mathcal{S}(A)$ of $A$, coincides with the weak*-topology.

\begin{definition}[\cite{R2,R4,R5,KD1,Long1}]
  If there exists a Lip-norm $L$ on a unital $C^*$-algebra $A$, we say that the pair $(A,L)$ is a \textit{compact quantum metric space}.
\end{definition}


\section{Metric structures of inductive limits}\label{sec2}

In this section, we will furnish the inductive limit of compact quantum metric spaces with a Lip-norm by means of the inverse limit of state spaces of building blocks.

Let $\{A_n\}$ be a sequence of unital $C^*$-algebras.  If for each $n\in\mathbb{N}$ there exists a unital $^*$-homomorphism $\phi_n :A_n\to A_{n+1}$, then $\{(A_n,\phi_n)\}$ is an inductive sequence of unital $C^*$-algebras, i.e.,
\[
A_1\xrightarrow{~\phi_1~}A_2\xrightarrow{~\phi_2~} A_3\xrightarrow{~\phi_3~} \cdots A_n\xrightarrow{~\phi_n~}A_{n+1}\xrightarrow{~\phi_{n+1}~}\cdots.
\]
It is well known that one can obtain a unital $C^*$-algebra $\varinjlim (A_n,\phi_n)$, the inductive limit (or direct limit) of the sequence $\{(A_n, \phi_n)\}$ \cite{Lin,M}. Furthermore, for each $n\in\mathbb{N}$ there is a unital $^*$-homomorphism $\phi^n:A_n\to \varinjlim (A_n,\phi_n)$ such that the diagram
\[
\xymatrix{
  A_n \ar@{->}[dr]_{\phi^n} \ar@{->}[r]^{\phi_n}
                & A_{n+1} \ar@{->}[d]^{\phi^{n+1}}  \\
                & \varinjlim (A_n,\phi_n)
               }
\]
commutes.

Note that if $ A$ is a unital $C^*$-algebra and $\{A_n\}$ is an increasing sequence of unital $C^*$-subalgebras of $ A$ whose union is dense in $ A$ and $\phi_n : A_n\to  A_{n+1}$ is the inclusion map, then $\{(A_n, \phi_n)\}$ is an inductive sequence of unital $C^*$-algebras and $ A$ is the inductive limit of $\{(A_n, \phi_n)\}$ \cite{Lin,M}, i.e.,
\[
A=\varinjlim (A_n,\phi_n)=\overline{\bigcup_{n=1}^{\infty}A_n}.
\]

The following example is a motivation for discussing the compact quantum metric space structures of inductive limits.

\begin{example}\label{motiv}
For any $n\in\mathbb{N}$, if $ A_n=C([0,1-\frac{1}{n+1}])$ is the $C^*$-algebra of all complex-valued continuous functions on $[0,1-\frac{1}{n+1}]$, and $L_n$ is the Lipchitz seminorm on $ A_n$ defined by
\[
L_n(f)=\sup\left\{\frac{|f(x)-f(y)|}{|x-y|}: x,y\in \left[0,1-\frac{1}{n+1}\right],x\neq y\right\},
\]
for $f\in A_n$, where the value $+\infty$ is permitted, we consider the Monge-Kantorovich metric $\rho_{L_n}$ on the state space $\mathcal{S}( A_n)$ of $A_n$ given by the formula
\[
\rho_{L_n}(\mu,\nu)=\sup\{|\mu(f)-\nu(f)|: f\in A_n,L_n(f)\leq 1\},
\]
for all $\mu,\nu\in \mathcal{S}(A_n)$. It is easy to see that $( A_n, L_n)$ is a compact quantum metric space.

Define a unital $^*$-homomorphism $\phi_n : A_n\to  A_{n+1}$ as follows:
\[
\phi_n(f)(t)=\begin{cases}
\begin{array}{ll}
f(t), &{\rm if}~~ t\in [0,1-\frac{1}{n+1}], \\
f(1-\frac{1}{n+1}),  & {\rm if}~~ t\in(1-\frac{1}{n+1},1-\frac{1}{n+2}].
\end{array}
\end{cases}
\]
Thus $\{(A_n, \phi_n)\}$ is an inductive sequence of unital $C^*$-algebras and
 $$
 (L_{n+1}\circ\phi_n)(f)=L_{n+1}(\phi_n(f))=L_n(f),
 $$
 for all $f\in A_n$.
Moreover, one can consider $A_n$ as a unital $C^*$-subalgebra of $C([0,1])$ by means of the following formula:
\[
\phi^n(f)(x)=\begin{cases}
\begin{array}{ll}
f(x), &{\rm if}~~ 0\leq x\leq 1-\frac{1}{n+1}, \\
f(1-\frac{1}{n+1}),  & {\rm if}~~ 1-\frac{1}{n+1}< x\leq 1,
\end{array}
\end{cases}
\]
for any $f\in A_n$, since $\phi^n$ is an isometric unital $^*$-homomorphism from $A_n$ to $C([0,1])$ for all $n\in\mathbb{N}$. In addition, the diagram
\[
\xymatrix{
  A_n \ar@{->}[dr]_{\phi^n} \ar@{->}[r]^{\phi_n}
                & A_{n+1} \ar@{->}[d]^{\phi^{n+1}}  \\
                & C([0,1])
               }
\]
is commutative for all $n\in\mathbb{N}$. It follows that
\[
C([0,1])= \varinjlim(A_n,\phi_n)=\overline{\bigcup_{n=1}^{\infty}C\left(\left[0,1-\frac{1}{n+1}\right]\right)}^{\|\cdot\|_{\infty}}.
\]

Now, we consider the Lipschitz seminorm $L$ on $C([0,1])$ given by
 \begin{equation}\label{lipconstant}
  L(f)=\sup\left\{\frac{|f(x)-f(y)|}{|x-y|}: x,y\in [0,1],x\neq y\right\},
 \end{equation}
 for all $f\in C([0,1])$.
Then it is obvious that
$$
 L_{n+1}(\phi_n(f))=L_n(f)=L(\phi^n(f)),
 $$
 for all $f\in A_n$.
 It is easy to confirm that the pair $(C([0,1]), L)$ is a compact quantum metric space by the Arzel\`{a}-Ascoli theorem \cite{Long3,OR,R1}. Furthermore, it follows from \cite[Theorem 2]{T} that the state space of $C([0,1])$ is affinely homeomorphic to the projective limit (or inverse limit) of the state spaces $\{\mathcal{S}( A_n)\}$.
\end{example}

\begin{definition}\label{inductivesequence}
Let $\{(A_n, L_n)\}$ be a sequence of compact quantum metric spaces.  If for each $n\in\mathbb{N}$ there exists a unital $^*$-homomorphism $\phi_n :A_n\to A_{n+1}$ such that $\phi_n(\mathrm{dom}(L_n))\subset\mathrm{dom}(L_{n+1})$, then we call $\{((A_n, L_n),\phi_n)\}$ an \textit{inductive sequence of compact quantum metric spaces}.
More precisely, we can represent it as
\[
(A_1, L_1)\xrightarrow{~\phi_1~}(A_2, L_2)\xrightarrow{~\phi_2~} (A_3, L_3)\xrightarrow{~\phi_3~}(A_4, L_4)\xrightarrow{~\phi_4~}\cdots.
\]
\end{definition}

Let $\{((A_n,L_n), \phi_n)\}$ be an inductive sequence of compact quantum metric spaces. Then $\{(A_n, \phi_n)\}$ is an inductive sequence of unital $C^*$-algebras, and hence get the inductive limit $\varinjlim (A_n,\phi_n)$.
Inspired by Example \ref{motiv}, we will endow $\varinjlim (A_n,\phi_n)$ with some compact quantum metric space structure in Theorem \ref{inductive}.

Let $\{((A_n, L_n),\phi_n)\}$ be an inductive sequence of compact quantum metric spaces. Then for any $n\in\mathbb{N}$, $(\mathcal{S}(A_n), \rho_{L_n})$ is a compact metric space  and from the unital $^*$-homomorphism  $\phi_n:A_n\to A_{n+1}$ we can get an affine continuous map $\hat{\phi}_n:\mathcal{S}(A_{n+1})\to \mathcal{S}(A_n)$ defined by
\[
\hat{\phi}_n(\mu)(a)=\mu(\phi_n(a)),\quad a\in A_n,
\]
for all $\mu\in\mathcal{S}(A_{n+1})$. As a result, we have the following inverse sequence $\{(\mathcal{S}(A_n),\hat{\phi}_n)\}$ of compact metric spaces:
\[
\mathcal{S}(A_1)\xleftarrow{~\hat{\phi}_1~}\mathcal{S}(A_2)\xleftarrow{~\hat{\phi}_2~} \mathcal{S}(A_3)\xleftarrow{~\hat{\phi}_3~} \mathcal{S}(A_4)\xleftarrow{~\hat{\phi}_4~}\cdots.
\]
Furthermore, we then obtain a Hausdorff space $\varprojlim(\mathcal{S}(A_n),\hat{\phi}_n)$, the inverse limit space of the sequence $\{(\mathcal{S}(A_n),\hat{\phi}_n)\}$,  which is given by
\[
\left\{(\mu_1,\mu_2,\mu_3,\ldots)\in \prod_{n=1}^\infty\mathcal{S}( A_n):\mu_n=\hat{\phi}_{n}(\mu_{n+1}),n\in\mathbb{N}\right\}.
\]
Since $\mathcal{S}(A_n)$ is a compact Hausdorff space under the weak $^*$-topology for all $n\in\mathbb{N}$, $\prod_{n=1}^\infty\mathcal{S}( A_n)$ is also compact and Hausdorff by Tychonoff's theorem, and hence $\varprojlim(\mathcal{S}(A_n),\hat{\phi}_n)$, as a closed subspace of $\prod_{n=1}^\infty\mathcal{S}( A_n)$, is a compact Hausdorff space.

In the sequel, we will denote $\varinjlim (A_n,\phi_n)$ and  $\varprojlim(\mathcal{S}(A_n),\hat{\phi}_n)$ by $A$ and $\mathcal{S}$, respectively, for the sake of simplicity of notations.

\begin{proposition}\label{affine}
  Let $\{((A_n, L_n),\phi_n)\}$ be an inductive sequence of compact quantum metric spaces. Then the state space $\mathcal{S}(A)$ of the inductive limit $A$ is affinely homeomorphic to the inverse limit $\mathcal{S}$ of the sequence $\{(\mathcal{S}(A_n),\hat{\phi}_n)\}$.
\end{proposition}

\begin{proof}
  For any $\mu\in \mathcal{S}(A)$, set
  \[
  \mu_n:=\mu\circ \phi^n,
  \]
  for all $n\in\mathbb{N}$. Then one can easily check that $(\mu_1,\mu_2,\mu_3,\ldots)$ is in $\mathcal{S}$ since
  \[
  \hat{\phi}_n(\mu_{n+1})=\mu_{n+1}\circ\phi_n=\mu\circ \phi^{n+1}\circ\phi_n=\mu\circ\phi^n=\mu_n,
  \]
  for all $n\in\mathbb{N}$. Thus we get a map
  \[
  \Phi: \mathcal{S}(A)\to \mathcal{S}
  \]
  by mapping $\mu\in\mathcal{S}(A)$ to $(\mu_1,\mu_2,\mu_3,\ldots)\in \mathcal{S}$. It is obvious that the map $\Phi$ is affine, continuous and injective.

  Now we just need to show that the map $\Phi$ is surjective. For any $(\mu_1,\mu_2,\mu_3,\ldots)\in\mathcal{S}$, we have
  \[
  (\mu_{n+1}\circ\phi_n)(a)=\hat{\phi}_{n}(\mu_{n+1})(a)=\mu_n(a),
  \]
for all $n\in\mathbb{N}$ and $a\in A_n$, and hence the following commutative diagram
$$\CD
  A_1 @>\phi_1>> A_2 @>\phi_2>> A_3 @>\phi_3>>\cdots A\\
  @V V \mu_1V @V  V\mu_2V @V  V\mu_3V  \\
  \mathbb{C} @=\mathbb{C} @=\mathbb{C} @= \cdots \mathbb{C}.
\endCD
$$
By Theorem 1.10.14 in \cite{Lin}, we obtain a state $\mu$ on $A$ with
\[
\mu(\phi^n(a))=\mu_n(a)
\]
for all $n\in \mathbb{N}$ and $a\in A_n$. For this $\mu$ we have $\Phi(\mu)=(\mu_1,\mu_2,\mu_3,\ldots)$, and so $\Phi$ is surjective. Therefore, $\Phi$ is an affine homeomorphism from the state space $\mathcal{S}(A)$ of $A$ onto the inverse limit $\mathcal{S}$ of $\{(\mathcal{S}(A_n),\hat{\phi}_n)\}$.
\end{proof}

Let $\{((A_n, L_n),\phi_n)\}$ be an inductive sequence of compact quantum metric spaces. Then for any $n\in\mathbb{N}$, $(\mathcal{S}(A_n), \rho_{L_n})$ is a compact metric space and its induced metric topology on the state space $\mathcal{S}(A_n)$ coincides with the weak $^*$-topology. Thus we can endow the inverse limit $\mathcal{S}$ with a product metric as
\[
\rho_0((\mu_1,\mu_2,\mu_3,\ldots),(\nu_1,\nu_2,\nu_3,\ldots))=\sum_{n=1}^\infty\frac{1}{2^{n}}\frac{\rho_{L_n}(\mu_n,\nu_n)}{1+\rho_{L_n}(\mu_n,\nu_n)},
\]
for all $(\mu_1,\mu_2,\mu_3,\ldots),(\nu_1,\nu_2,\nu_3,\ldots)\in\mathcal{S}$. It follows that for this metric the inverse limit $\mathcal{S}$ is a compact metric space and its induced metric topology on $\mathcal{S}$ coincides with the product topology, as a subspace of $\prod_{n=1}^\infty\mathcal{S}( A_n)$. By Proposition \ref{affine} we can equip the state space $\mathcal{S}(A)$ of the inductive limit $A$ with the metric $\rho$ as follows:
\begin{equation*}\label{product}
  \rho(\mu,\nu):=\rho_0(\Phi(\mu),\Phi(\nu))=\sum_{n=1}^\infty\frac{1}{2^{n}}\frac{\rho_{L_n}(\mu_n,\nu_n)}{1+\rho_{L_n}(\mu_n,\nu_n)},
\end{equation*}
for all $\mu,\nu\in\mathcal{S}(A)$, where $\Phi$ is the affine homeomorphism from $\mathcal{S}(A)$ onto $\mathcal{S}$. Therefore, we have the following proposition.

\begin{proposition}\label{weak}
 Let $\{((A_n, L_n),\phi_n)\}$ be an inductive sequence of compact quantum metric spaces. Then $\left(\mathcal{S}(A),\rho\right)$ is a compact metric space, and the metric topology on the state space $\mathcal{S}(A)$, induced by $\rho$, agrees with the weak $^*$-topology.
\end{proposition}

\begin{proof}
  The conclusion follows from Proposition \ref{affine} and the discussion before the proposition.
\end{proof}

Note that for the seminorm $L$ on $C([0,1])$ given by the equation \eqref{lipconstant}, we have

\begin{equation}\label{lipconstant1}
  L(f)=\sup\left\{\frac{|\mu(f)-\nu(f)|}{\rho_L(\mu,\nu)}: \mu,\nu\in\mathcal{S}(C([0,1])),\mu\neq \nu\right\},
\end{equation}
  for all $f\in C([0,1])$. Indeed, it is obvious that
 \[
 L(f)\leq \sup\left\{\frac{|\mu(f)-\nu(f)|}{\rho_L(\mu,\nu)}: \mu,\nu\in\mathcal{S}(C([0,1])),\mu\neq \nu\right\},
 \]
 for all $f\in C([0,1])$. It follows from the definition of $\rho_L$ that for any $f\in C([0,1])$, we have
 \[
 |\mu(f)-\nu(f)|\leq L(f)\rho_L(\mu,\nu),
 \]
 for all $\mu,\nu\in\mathcal{S}(C([0,1]))$. Thus
 \[
  \sup\left\{\frac{|\mu(f)-\nu(f)|}{\rho_L(\mu,\nu)}: \mu,\nu\in\mathcal{S}(C([0,1])),\mu\neq \nu\right\}\leq L(f),
 \]
 as desired.

 Now we are ready to define a seminorm on the inductive limit of an inductive sequence of compact quantum metric spaces.

Let $\{((A_n, L_n),\phi_n)\}$ be an inductive sequence of compact quantum metric spaces. Now, as the equation \eqref{lipconstant1}
,  we can define a seminorm $L$ on the inductive limit $A$ by the formula
\begin{equation}\label{seminorm}
  L(a)=\sup\left\{\frac{|\mu(a)-\nu(a)|}{\rho(\mu,\nu)}:\mu,\nu\in\mathcal{S}(A),\mu\neq\nu\right\},
\end{equation}
for all $a\in A$.

In the sequel, we will always consider the seminorm $L$, as in equation \eqref{seminorm}, for the inductive limit of an inductive sequence of compact quantum metric spaces if there is no other specific instruction.

\begin{proposition}\label{lipschitz}
Let $\{((A_n, L_n),\phi_n)\}$ be an inductive sequence of compact quantum metric spaces. Then $L$ is a lower semicontinuous Lipschitz seminorm on the inductive limit $A$.
\end{proposition}

\begin{proof}
It is easy to check that  $L(a^*)=L(a)$ for all $a\in A$.
Since the state space $\mathcal{S}(A)$ separates the elements of $A$, we conclude that $L(a)=0$ if and only if $a\in\mathbb{C}\mathbf{1}_A$.
The lower semicontinuity of $L$ is straightforward.

 For any $n\in\mathbb{N}$ and  $a\in A_n$ with $L_n(a)<\infty$, we have that $L(\phi^n(a))<\infty$.
 Indeed, for any $\mu,\nu\in\mathcal{S}(A)$ with $\mu\neq \nu$, we have
 \[
 |\mu(\phi^n(a))-\nu(\phi^n(a))|=|\mu_n(a)-\nu_n(a)|\leq L_n(a)\rho_{L_n}(\mu_n,\nu_n)
 \]
and
 \[
\rho(\mu,\nu)=\sum_{n=1}^\infty\frac{1}{2^{n}}\frac{\rho_{L_n}(\mu_n,\nu_n)}{1+\rho_{L_n}(\mu_n,\nu_n)}
\geq \frac{1}{2^{n}}\frac{\rho_{L_n}(\mu_n,\nu_n)}{1+\rho_{L_n}(\mu_n,\nu_n)},
\]
hence
\begin{align*}
  \frac{|\mu(\phi^n(a))-\nu(\phi^n(a))|}{\rho(\mu,\nu)} & \leq L_n(a)\rho_{L_n}(\mu_n,\nu_n)\times 2^n\frac{1+\rho_{L_n}(\mu_n,\nu_n)}{\rho_{L_n}(\mu_n,\nu_n)} \\
   &= 2^nL_n(a)(1+\rho_{L_n}(\mu_n,\nu_n))\\
   &\leq 2^nL_n(a)(1+\mathrm{diam}(\mathcal{S}(A_n),\rho_{L_n})),
\end{align*}
where
\[
\mathrm{diam}(\mathcal{S}(A_n),\rho_{L_n})=\sup\{\rho_{L_n}(\mu',\nu'):\mu',\nu'\in\mathcal{S}(A_n)\}
\]
is the diameter of $(\mathcal{S}(A_n),\rho_{L_n})$. It follows that
\begin{align*}
  L(\phi^n(a)) &=\sup\left\{\frac{|\mu(\phi^n(a))-\nu(\phi^n(a))|}{\rho(\mu,\nu)}:\mu,\nu\in\mathcal{S}(A),\mu\neq\nu\right\}  \\
   & \leq 2^nL_n(a)(1+\mathrm{diam}(\mathcal{S}(A_n),\rho_{L_n}))<\infty.
\end{align*}
Thus
$$
\bigcup_{n=1}^\infty\phi^n(\mathrm{dom}(L_n))\subset \mathrm{dom}(L).
$$
This implies that $\mathrm{dom}(L)$ is a dense subspace of $A$, and completes the proof.
\end{proof}

\begin{theorem}\label{inductive}
 Let $\{((A_n, L_n),\phi_n)\}$ be an inductive sequence of compact quantum metric spaces. Then $L$ is a Lip-norm on the inductive limit $A$, i.e., $(A,L)$ is a compact quantum metric space.
\end{theorem}

\begin{proof}
From Proposition \ref{lipschitz}, we see that $L$ is a Lipschitz seminorm on $A$. Let $\mu\in\mathcal{S}(A)$.  By Proposition 1.3 in \cite{OR} or Proposition 6.8 in \cite{Long3}, we just need to show that the set
\[
\mathcal{B}=\{a\in A:L(a)\leq 1,\mu(a)=0\}
\]
 is a norm totally bounded subset of $A$.

By Proposition \ref{weak}, $(\mathcal{S}(A),\rho)$ is a compact metric space and the $\rho$-topology on $\mathcal{S}(A)$ agrees with the weak $^*$-topology. For any $a\in \mathcal{B}$ and $\nu_1,\nu_2\in\mathcal{S}(A)$, we have
 \[
 |\hat{a}(\nu_1)-\hat{a}(\nu_2)|=|\nu_1(a)-\nu_2(a)|\leq L(a)\rho(\nu_1,\nu_2)\leq \rho(\nu_1,\nu_2),
 \]
 and hence $\hat{\mathcal{B}}$ is a family of equicontinuous continuous functions on $(\mathcal{S}(A),\rho)$.

 For any $a\in \mathcal{B}$, we have
 \[
 |\hat{a}(\nu)|=|\nu(a)-\mu(a)|\leq L(a)\rho(\mu,\nu)\leq \mathrm{diam}(\mathcal{S}(A),\rho)
 \]
 for all $\nu\in\mathcal{S}(A)$, and so
 \[
 \|\hat{a}\|_{\infty}\leq \mathrm{diam}(\mathcal{S}(A),\rho).
 \]
 It follows that $\hat{\mathcal{B}}$ is a bounded subset of the unital $C^*$-algebra $C(\mathcal{S}(A))$ of complex-valued continuous functions on $\mathcal{S}(A)$.
 By Arzel\`{a}-Ascoli theorem, $\hat{\mathcal{B}}$ is a totally bounded subset of $C(\mathcal{S}(A))$.

From Kadison representation theorem, the canonical map
$$a\in A_{sa}\mapsto \hat{a}\in \mathrm{Aff}(\mathcal{S}(A))\subset C(\mathcal{S}(A)),$$
where $\mathrm{Aff}(\mathcal{S}(A))$ is the set of all real-valued affine continuous functions on $\mathcal{S}(A)$, is a unital order isomorphism, and hence an isometry. For any $a\in A$, we have
\begin{eqnarray*}
\|\hat{a}\|_\infty\leq\|a\| &\leq& \|a_1\|+\|a_2\| =\|\widehat{a_1}\|_\infty+\|\widehat{a_2}\|_\infty  \\
     &\leq& \|\hat{a}\|_\infty+\|\hat{a}\|_\infty=2\|\hat{a}\|_\infty,
\end{eqnarray*}
where $a_1=\frac{a+a^*}{2}, a_2=\frac{a-a^*}{2i}\in A_{sa}$. For any $\varepsilon>0$, since $\hat{\mathcal{B}}$ is totally  bounded, there exist $a_1, a_2,\ldots,a_m\in \mathcal{B}$ such that for any $\hat{a}\in\hat{\mathcal{B}}$, there is an $a_i$ such that $\|\hat{a}-\hat{a}_i\|_\infty<\varepsilon/2$. It follows that for any $c\in\mathcal{B}$, there is an $a_i$ such that $\|\hat{c}-\hat{a}_i\|_\infty<\frac{\varepsilon}{2}$, and thus $\|c-a_i\|\leq 2\|\hat{c}-\hat{a}_i\|_\infty<\varepsilon$. This implies that $\mathcal{B}$ is a totally bounded subset of $A$ in the norm topology, and completes the proof of the theorem.
\end{proof}

In \cite{A1}, Aguilar investigated compact quantum metrics on the unital inductive limit $A$ arising from an increasing sequence $\{A_n\}$ of untial $C^*$-subalgebras of $A$. Moreover, if $\{(A_n,L_n)\}$ is a sequence of $(C,D)$-quasi-Leibniz compact quantum metric space, and for each $n\in \{0\}\cup\mathbb{N}$:
\begin{enumerate}
\item $\mathrm{dom}(L_n)=\{a\in A_n:L_{n}(a)<\infty\}$ is a dense $^*$-subalgebra of $A_n$;
  \item $L_{n+1}(a)\leq L_n(a)$ for all $a\in A_n$;
  \item there exists a positive real number sequence $\{\beta_n\}_{n=0}^\infty$ such that  $\sum_{n=0}^{\infty}\beta_n<\infty$ and the length of the bridge $\gamma_{n,n+1}=(A_{n+1},1_A,\iota_{n,n+1},\mathrm{id}_{n+1})$ satisfies
      \[
      \lambda(\gamma_{n,n+1}|L_n,L_{n+1})\leq \beta_n,
      \]
\end{enumerate}
then he gave a compact quantum metric on the inductive limit $A$ by utilizing the completeness of the dual Gromov-Hausdorff propinquity (see Definitions 1.2, 1.6, 1.7, 2.5, Lemma 1.8 and Theorem 2.15 in \cite{A1}). In particular, this gives a compact quantum metric on unital AF algebras (see Theorem 3.4 in \cite{A1}).
In \cite{ALR}, Aguilar, Latr\'{e}moli\`{e}re and Rainone took a similar assumption and approach to show the existence of compact quantum metrics on the Bunce-Deddens algebras by means of quantum Gromov-Hausdorff distance (see Theorem 6.15 in \cite{ALR}).

In Definition \ref{inductivesequence} and Theorem \ref{inductive}, we use a different approach and some more relaxed assumptions to obtain a compact quantum metric space structure on the unital inductive limit  by means of the inverse limit of state spaces of building blocks.

The following important examples of compact quantum metric spaces are immediate.

\begin{corollary}[\cite{A1,ALR}]\label{cor3.7}
 UHF algebras, unital AF-algebras and Bunce-Deddens algebras are compact quantum metric spaces.
\end{corollary}

\begin{example}
Bratteli and Kishimoto showed that the ``non commutative sphere'' $B_{\theta}$ is actually an $AF$ C$^*$-algebra \cite{BK}, when $\theta$ is irrational, and hence by Corollary \ref{cor3.7} every non commutative sphere is a compact quantum metric space.
\end{example}


\section{Lipschitz isomorphisms of inductive limits}

In this section, we will investigate the relations of compact quantum metric space structures between two  inductive limits by way of their building blocks, and give some sufficient conditions that two inductive limits of compact quantum metric spaces are Lipschitz isomorphic.

\begin{proposition}\label{universal}
 Let $\{((A_n, L_n),\phi_n)\}$ be an inductive sequence of compact quantum metric spaces.  Suppose that $(B,L_B)$ is a compact quantum metric space with lower semicontinuous Lipschitz seminorm, and that for each $n\in\mathbb{N}$,  there are a positive real number $\lambda_n>0$ and a unital $^*$-homomorphism $\psi_n:A_n\to B$ such that the  diagram
\[
\xymatrix{
  A_n \ar@{->}[dr]_{\psi_n} \ar@{->}[r]^{\phi_n}
                & A_{n+1} \ar@{->}[d]^{\psi_{n+1}}  \\
                & B
               }
\]
commutes and
\[
L_B(\psi_n(a))\leq \lambda_nL_n(a)
\]
for all $a\in A_n$. If the sequence $\{\lambda_n\}$ is bounded, then there is a unique unital $^*$-homomorphism $\psi:A\to B$ such that for each $n\in\mathbb{N}$, the diagram
\[
\xymatrix{
  A_n \ar@{->}[dr]_{\psi_n} \ar@{->}[r]^{\phi^n}
                & A \ar@{->}[d]^{\psi}  \\
                & B
               }
\]
commutes and
\[
L_B(\psi(a))\leq 2\sup\{\lambda_n:n\in\mathbb{N}\}L(a)
\]
for all $a\in A$.
\end{proposition}

\begin{proof}
From Theorem 6.1.2 in \cite{M}, we see that there exists a unique unital $^*$-homomorphism $\psi:A\to B$ such that for each $n\in\mathbb{N}$, the diagram
\[
\xymatrix{
  A_n \ar@{->}[dr]_{\psi_n} \ar@{->}[r]^{\phi^n}
                & A \ar@{->}[d]^{\psi}  \\
                & B
               }
\]
commutes. Therefore, we just need to show that \[
L_B(\psi(a))\leq 2\sup\{\lambda_n:n\in\mathbb{N}\}L(a),
\]
for all $a\in A$ with $L(a)<\infty$.

For any $n\in\mathbb{N}$ and $\mu,\nu\in\mathcal{S}(B)$, we have
\begin{eqnarray*}
\rho_{L_n}(\hat{\psi}_n(\mu),\hat{\psi}_n(\nu))  &=&\sup\{|\hat{\psi}_n(\mu)(a)-\hat{\psi}_n(\nu)(a)|:a\in A_n,L_n(a)\leq1\}  \\
     &=& \sup\{|\mu(\psi_n(a))-\nu(\psi_n(a))|:a\in A_n,L_n(a)\leq 1\} \\
     &\leq& \lambda_n\sup\{|\mu(b)-\nu(b)|:b\in B,L_B(b)\leq 1\}  \\
     &=& \lambda_n \rho_{L_B}(\mu,\nu).
\end{eqnarray*}
It follows that
\begin{align*}
  \rho(\hat{\psi}(\mu),\hat{\psi}(\nu))&=\sum_{n=1}^\infty\frac{1}{2^{n}}\frac{\rho_{L_n}(\hat{\psi}(\mu)_n,\hat{\psi}(\nu)_n)}{1+\rho_{L_n}(\hat{\psi}(\mu)_n,\hat{\psi}(\nu)_n)}\\
   &=\sum_{n=1}^\infty\frac{1}{2^{n}}\frac{\rho_{L_n}(\hat{\psi}_n(\mu),\hat{\psi}_n(\nu))}{1+\rho_{L_n}(\hat{\psi}_n(\mu),\hat{\psi}_n(\nu))}  \\
   &\leq \sum_{n=1}^\infty\frac{1}{2^{n}}\frac{\lambda_n \rho_{L_B}(\mu,\nu)}{1+\lambda_n \rho_{L_B}(\mu,\nu)} \\
   &\leq \sum_{n=1}^\infty\frac{1}{2^{n}}\frac{\lambda \rho_{L_B}(\mu,\nu)}{1+\lambda \rho_{L_B}(\mu,\nu)}   \\
   &\leq \sum_{n=1}^\infty\frac{\lambda}{2^{n}} \rho_{L_B}(\mu,\nu)\\
   &=\lambda \rho_{L_B}(\mu,\nu),
\end{align*}
for all $\mu,\nu\in\mathcal{S}(B)$, where
\[
\lambda=\sup\{\lambda_n:n\in\mathbb{N}\}<\infty.
\]
As a result, for any $\mu,\nu\in\mathcal{S}(B)$ and $a\in A$ with $L(a)<\infty$, we have
\begin{align*}
  |\mu(\psi(a))-\nu(\psi(a))| &=|\hat{\psi}(\mu)(a)-\hat{\psi}(\nu)(a)|  \\
   &\leq L(a)\rho(\hat{\psi}(\mu),\hat{\psi}(\nu))  \\
   & \leq\lambda  L(a)\rho_{L_B}(\mu,\nu),
\end{align*}
and hence
\begin{align*}
  L_{\rho_{L_B}}(\psi(a)) &:=\sup\left\{\frac{|\mu(\psi(a))-\nu(\psi(a))|}{\rho_{L_B}(\mu,\nu)}:\mu,\nu\in\mathcal{S}(B),\mu\neq\nu\right\}  \\
   &\leq\lambda  L(a)<\infty,
\end{align*}
for all $a\in A$ with $L(a)<\infty$.
Consequently, by Theorem 4.1 in \cite{R2}, for any $a\in A$ with $L(a)<\infty$ we have
\begin{align*}
 L_B(\psi(a))&=L_B(\psi(a_1)+\psi(a_2))\\
 &\leq L_B(\psi(a_1))+L_B(\psi(a_2))\\
   &= L_{\rho_{L_B}}(\psi(a_1))+L_{\rho_{L_B}}(\psi(a_2))  \\
   &\leq\lambda  L(a_1)+\lambda  L(a_2)   \\
   &\leq 2\lambda  L(a)<\infty,
\end{align*}
where $a_1=\frac{a+a^*}{2},a_2=\frac{a-a^*}{2i}\in A_{sa}$. This completes the proof.
\end{proof}

\begin{corollary}\label{blo-ind}
  Let $\{((A_n, L_{n}),\phi_n)\}$ be an inductive sequence of compact quantum metric space. If there are a lower semicontinuous Lip-norm $L_A$ on $A$ and a bounded sequence $\{\gamma_n\}$ such that
\[
L_A(\phi^{n}(a))\leq \gamma_nL_{n}(a)
\]
 for all $n\in\mathbb{N}$ and $a\in A_n$, then
\[
L_A(a)\leq 2\sup\{\gamma_n:n\in\mathbb{N}\}L(a)
\]
for all $a\in A$.
\end{corollary}

\begin{proof}
 The conclusion follows from Proposition \ref{universal}.
\end{proof}

\begin{corollary}
  Let $\{((A_n, L_{n}),\phi_n)\}$ and $\{((A_n, L_{n}'),\phi_n)\}$ be two inductive sequences of compact quantum metric spaces. If there is a bounded sequence $\{\gamma_n\}$ such that
\[
L'(\phi^{n}(a))\leq \gamma_nL_{n}(a)
\]
 for all  $n\in\mathbb{N}$ and $a\in A_n$, then
\[
L'(a)\leq 2\sup\{\gamma_n:n\in\mathbb{N}\}L(a)
\]
for all $a\in A$.
\end{corollary}

\begin{proof}
 The conclusion follows from Proposition \ref{lipschitz} and  Corollary \ref{blo-ind}.
\end{proof}

Immediately, we have the following:

\begin{corollary}
  Let $((A_n, L_{n}),\phi_n)_{n\geq 1}$ and $((A_n, L_{n}'),\phi_n)_{n\geq 1}$ be two inductive sequences of compact quantum metric spaces. If there are bounded sequences  $\{\lambda_n\}$ and $\{\gamma_n\}$ such that
\[
L'(\phi^{n}(a))\leq \gamma_nL_{n}(a)
\]
and
\[
 L(\phi^{n}(a))\leq \lambda_nL_{n}'(a)
\]
 for all  $n\in\mathbb{N}$ and $a\in A_n$, respectively, then
\[
\frac{1}{2\sup\{\gamma_n:n\in\mathbb{N}\}}L'(a)\leq L(a)\leq 2\sup\{\lambda_n:n\in\mathbb{N}\}L'(a)
\]
for all $a\in A$.
\end{corollary}

In the following two propositions, the seminorm $L_B$ on $B$ does not need to be the seminorm in the equation \eqref{seminorm}.

\begin{proposition}\label{commutative}
Let $\{((A_n, L_{A_n}),\phi_n)\}$ and $\{((B_n, L_{B_n}),\varphi_n)\}$ be two inductive sequences of compact quantum metric spaces. Suppose that for each $n\in\mathbb{N}$,  there is a unital $^*$-homomorphism $\psi_n:A_n\to B_n$ such that the diagram
  $$\CD
  A_1 @>\phi_1>> A_2 @>\phi_2>> A_3 @>\phi_3>>\cdots A\\
  @V V \psi_1V @V  V\psi_2V @V  V\psi_3V  \\
  B_1 @>\varphi_1>> B_2@>\varphi_2>> B_3 @>\varphi_3>> \cdots B
\endCD
$$
is commutative.
 If there are  a bounded sequence $\{\gamma_n\}$ and a lower semicontinuous Lip-norm $L_B$ on $B$ such that
\[
L_B((\varphi^{n}\circ\psi_n)(a))\leq \gamma_nL_{A_n}(a)
\]
 for all $n\in\mathbb{N}$ and $a\in A_n$, then there is a unique unital $^*$-homomorphism $\psi:A\to B$ such that for each $n\in\mathbb{N}$, the diagram
\[
\xymatrix{
  A_n \ar@{->}[d]_{\psi_n} \ar@{->}[r]^{\phi^n}
                & A \ar@{->}[d]^{\psi}  \\
    B_n \ar@{->}[r]_{\varphi^n}            & B
               }
\]
is commutative and
\[
L_B(\psi(a))\leq 2\sup\{\gamma_n:n\in\mathbb{N}\}L_A(a)
\]
for all $a\in A$.
\end{proposition}

\begin{proof}
  For any $n\in\mathbb{N}$, we can get a commutative diagram as follows:
\[
\xymatrix{
  A_n \ar@{->}[dr]_{\varphi^{n}\circ\psi_n} \ar@{->}[r]^{\phi_n}
                & A_{n+1} \ar@{->}[d]^{\varphi^{n+1}\circ\psi_{n+1}}  \\
                & B
               }
\]
Moreover, for any $n\in\mathbb{N}$ and $a\in A_n$ we have
\[
L_B((\varphi^{n}\circ\psi_n)(a))\leq \gamma_nL_{A_n}(a).
\]
Since the sequence $\{\gamma_n\}$ is bounded, by Proposition \ref{universal}, there is a unique unital $^*$-homomorphism $\psi:A\to B$ such that for each $n\in\mathbb{N}$, the diagram
\[
\xymatrix{
  A_n \ar@{->}[d]_{\psi_n} \ar@{->}[r]^{\phi^n}
                & A \ar@{->}[d]^{\psi}  \\
    B_n \ar@{->}[r]_{\varphi^n}            & B
               }
\]
is commutative and
\[
L_B(\psi(a))\leq 2\sup\{\gamma_n:n\in\mathbb{N}\}L_A(a)
\]
for all $a\in A$.
\end{proof}

\begin{proposition}
Let $\{((A_n, L_{A_n}),\phi_n)\}$ and $\{((B_n, L_{B_n}),\varphi_n)\}$ be two inductive sequences of compact quantum metric spaces. Suppose that for each $n\in\mathbb{N}$, there are a positive real number $\lambda_n>0$ and a unital $^*$-homomorphism $\psi_n:A_n\to B_n$ such that the diagram
  $$\CD
  A_1 @>\phi_1>> A_2 @>\phi_2>> A_3 @>\phi_3>>\cdots A\\
  @V V \psi_1V @V  V\psi_2V @V  V\psi_3V  \\
  B_1 @>\varphi_1>> B_2@>\varphi_2>> B_3 @>\varphi_3>> \cdots B
\endCD
$$
is commutative and
\[
L_{B_n}(\psi_n(a))\leq \lambda_nL_{A_n}(a)
\]
for all $a\in A_n$.
If the sequence $\{\lambda_n\}$ is bounded and there are a bounded sequence $\{\gamma_n\}$ and a lower semicontinuous Lip-norm $L_B$ on $B$ such that
\[
L_{B}(\varphi^n(b))\leq \gamma_nL_{B_n}(b)
\]
for all $b\in B_n$, then there is a unique unital $^*$-homomorphism $\psi:A\to B$ such that for each $n\in\mathbb{N}$, the diagram
\[
\xymatrix{
  A_n \ar@{->}[d]_{\psi_n} \ar@{->}[r]^{\phi^n}
                & A \ar@{->}[d]^{\psi}  \\
    B_n \ar@{->}[r]_{\varphi^n}            & B
               }
\]
is commutative and
\[
L_B(\psi(a))\leq 2\sup\{\lambda_n\gamma_n:n\in\mathbb{N}\}L_A(a)
\]
for all $a\in A$.
\end{proposition}

\begin{proof}
 From Theorem 1.10.14 in \cite{Lin}, we see that there is a unique unital $^*$-homomorphism $\psi:A\to B$ such that for each $n\in\mathbb{N}$, the diagram
\[
\xymatrix{
  A_n \ar@{->}[d]_{\psi_n} \ar@{->}[r]^{\phi^n}
                & A \ar@{->}[d]^{\psi}  \\
    B_n \ar@{->}[r]_{\varphi^n}            & B
               }
\]
is commutative. Therefore, we just need to show that \[
L_B(\psi(a))\leq 2\sup\{\lambda_n\gamma_n:n\in\mathbb{N}\}L_A(a),
\]
for all $a\in A$ with $L_A(a)<\infty$.

Note that for any $n\in\mathbb{N}$ and $a\in A_n$, we have
\begin{align*}
  L_B((\varphi^{n}\circ\psi_n)(a)) &=L_B(\varphi^{n}(\psi_n(a)))  \\
   &\leq \gamma_nL_{B_n}(\psi_n(a))\leq \lambda_n\gamma_nL_{A_n}(a)
\end{align*}
and
\[
\sup\{\lambda_n\gamma_n:n\in\mathbb{N}\}\leq \sup\{\lambda_n:n\in\mathbb{N}\}\sup\{\gamma_n:n\in\mathbb{N}\}<\infty.
\]
Thus, the sequence $\{\lambda_n\gamma_{n}\}$ is bounded, and we obtain the conclusion by Proposition \ref{commutative}, as desired.
\end{proof}

Let $(A,L_A)$ and $(B,L_B)$  be two compact quantum metric spaces, if there is a unital $^*$-isomorphism $\Psi:A\to B$ from $A$ onto $B$ such that
\[
L_B(\Psi(a))\leq \lambda L_A(a),\quad L_A(\Psi^{-1}(b))\leq \gamma L_B(b),
\]
for some constants $\lambda,\gamma>0$ and all $a\in A$ and $b\in B$, we say that $\Psi$ is a \textit{Lipschitz isomorphism} from $(A,L_A)$ onto $(B,L_B)$, and that $(A,L_A)$ and $(B,L_B)$ are \textit{Lipschitz isomorphic} \cite{KD1,Long4,Long5,Long6,Wu2011}.

\begin{theorem}\label{isomorphism}
Let $\{((A_n, L_{A_n}),\phi_n)\}$ and $\{((B_n, L_{B_n}),\varphi_n)\}$ be two inductive sequences of compact quantum metric spaces.  Suppose that for each $n\in\mathbb{N}$, there are unital $^*$-homomorphisms $\psi_n: A_n\to B_n$ and $\Psi_n:B_n\to A_{n+1}$ such that the diagram
\[\xymatrix{
  A_1\ar[d]^{\psi_1} \ar[r]^{\phi_1} & A_2 \ar[d]^{\psi_2} \ar[r]^{\phi_2} & A_3 \ar[d]^{\psi_3} \ar[r]^{\phi_3} & \cdots A \\
  B_1 \ar[r]^{\varphi_1}\ar[ru]_{\Psi_1} & B_2 \ar[r]^{\varphi_2} \ar[ru]_{\Psi_2}& B_3\ar[ru]_{\Psi_3} \ar[r]^{\varphi_3} & \cdots B   }
  \]
  is commutative.
  If there are bounded sequences $\{\lambda_n\}$ and $\{\gamma_n\}$ such that
  \[
L_B((\varphi^{n}\circ\psi_n)(a))\leq \lambda_nL_{A_n}(a)
\]
 for all $a\in A_n$ and
 \[
L_A((\phi^{n+1}\circ\Psi_n)(a))\leq \gamma_nL_{B_n}(b)
\]
 for all $b\in B_n$, respectively, then there are unique unital $^*$-homomorphisms $\psi : A\to B$ and $ \Psi : B\to A$ such that
  \begin{enumerate}
    \item $\psi \circ \Psi = \mathrm{id}_B$ and $\Psi \circ \psi = \mathrm{id}_A$;
    \item for each $n\in\mathbb{N}$ the diagram
  \[
  \xymatrix{
    A_n \ar[d]^{\psi_n} \ar[rrr]^{\phi^n} &&& A\ar@<.8mm>@{.>}[d]^{\psi} \\
    B_n \ar[rrr]_{\varphi^n}\ar[rrru]_{\phi^{n+1}\circ\Psi_n}&& & B \ar@<.8mm>@{.>}[u]^{\Psi} }
  \]
 is commutative;
    \item for any $a\in A$,
    \[
L_B(\psi(a))\leq 2\sup\{\lambda_n:n\in\mathbb{N}\}L_A(a);
\]
\item  for any $b\in B$,
\[
 L_A(\Psi(b))\leq 2\sup\{\gamma_n:n\in\mathbb{N}\}L_B(b).
\]
  \end{enumerate}
Therefore, the compact quantum metric spaces $(A,L_A)$ and $(B,L_B)$ are Lipschitz isomorphic.
\end{theorem}

\begin{proof}
From Theorem 1.10.16 in \cite{Lin}, we see that there are unital $^*$-homomorphisms $\psi : A\to B$ and $ \Psi : B\to A$ such that $\psi \circ \Psi = \mathrm{id}_B$ and $\Psi \circ \psi = \mathrm{id}_A$, and such that for each $n\in\mathbb{N}$, the diagram
   \[
  \xymatrix{
    A_n \ar[d]^{\psi_n} \ar[rrr]^{\phi^n} &&& A\ar@<.8mm>@{.>}[d]^{\psi} \\
    B_n \ar[rrr]_{\varphi^n}\ar[rrru]_{\phi^{n+1}\circ\Psi_n}&& & B \ar@<.8mm>@{.>}[u]^{\Psi} }
  \]
  is commutative. Therefore, we just need to show that \[
L_B(\psi(a))\leq 2\sup\{\lambda_n:n\in\mathbb{N}\}L_A(a),
\]
for all $a\in A$ and
\[
\quad L_A(\Psi(b))\leq 2\sup\{\gamma_n:n\in\mathbb{N}\}L_B(b),
\]
for all $b\in B$. However, these two statements follow from Proposition \ref{commutative}.
\end{proof}

\begin{corollary}\label{isomorphism1}
Let $\{((A_n, L_{A_n}),\phi_n)\}$ and $\{((B_n, L_{B_n}),\varphi_n)\}$ be two inductive sequences of compact quantum metric spaces.  Suppose that for each $n\in\mathbb{N}$, there are unital $^*$-homomorphisms $\psi_n: A_n\to B_n$ and $\Psi_n:B_n\to A_{n+1}$ such that the diagram
\[\xymatrix{
  A_1\ar[d]^{\psi_1} \ar[r]^{\phi_1} & A_2 \ar[d]^{\psi_2} \ar[r]^{\phi_2} & A_3 \ar[d]^{\psi_3} \ar[r]^{\phi_3} & \cdots A \\
  B_1 \ar[r]^{\varphi_1}\ar[ru]_{\Psi_1} & B_2 \ar[r]^{\varphi_2} \ar[ru]_{\Psi_2}& B_3\ar[ru]_{\Psi_3} \ar[r]^{\varphi_3} & \cdots B   }
  \]
  is commutative.
  If there are bounded sequences $\{\lambda_n\}$, $\{\gamma_n\}$, $\{\alpha_n\}$ and $\{\beta_n\}$ such that
  \begin{gather*}
    L_{B_n}(\psi_n(a))\leq \lambda_nL_{A_n}(a), \quad  L_{B}(\varphi^n(a))\leq \gamma_nL_{B_n}(b),  \\
   L_{A_{n+1}}(\Psi_n(a))\leq \alpha_nL_{B_n}(b),  \quad   L_{A}(\phi^n(a))\leq \beta_nL_{A_n}(a),
  \end{gather*}
 for all $a\in A_n$, $b\in B_n$, respectively, then there are unital $^*$-homomorphisms $\psi : A\to B$ and $ \Psi : B\to A$ such that
  \begin{enumerate}
    \item $\psi \circ \Psi = \mathrm{id}_B$ and $\Psi \circ \psi = \mathrm{id}_A$;
    \item for each $n\in\mathbb{N}$ the diagram
  \[
  \xymatrix{
    A_n \ar[d]^{\psi_n} \ar[rrr]^{\phi^n} &&& A\ar@<.8mm>@{.>}[d]^{\psi} \\
    B_n \ar[rrr]_{\varphi^n}\ar[rrru]_{\phi^{n+1}\circ\Psi_n}&& & B \ar@<.8mm>@{.>}[u]^{\Psi} }
  \]
 is commutative;
    \item for any $a\in A$,
    \[
L_B(\psi(a))\leq 2\sup\{\lambda_n\gamma_n:n\in\mathbb{N}\}L_A(a);
\]
\item  for any $b\in B$,
\[
 L_A(\Psi(b))\leq 2\sup\{\alpha_n\beta_{n+1}:n\in\mathbb{N}\}L_B(b).
\]
  \end{enumerate}
\end{corollary}

\begin{proof}
Note that for any $n\in\mathbb{N}$ and $a\in A_n,b\in B_n$, we have
\begin{align*}
  L_B((\varphi^{n}\circ\psi_n)(a)) &=L_B(\varphi^{n}(\psi_n(a)))  \\
   &\leq \gamma_nL_{B_n}(\psi_n(a))\leq \lambda_n\gamma_nL_{A_n}(a)
\end{align*}
and
\begin{align*}
  L_A((\phi^{n+1}\circ\Psi_n)(b)) &=L_A(\phi^{n+1}(\Psi_n(b)))  \\
   &\leq \beta_{n+1}L_{A_{n+1}}(\Psi_n(b))\leq \alpha_n\beta_{n+1}L_{B_n}(b),
\end{align*}
respectively. Moreover, it is easy to find that
\[
\sup\{\lambda_n\gamma_n:n\in\mathbb{N}\}\leq \sup\{\lambda_n:n\in\mathbb{N}\}\sup\{\gamma_n:n\in\mathbb{N}\}<\infty
\]
and
\begin{align*}
 \sup\{\alpha_n\beta_{n+1}:n\in\mathbb{N}\}  &\leq \sup\{\alpha_n:n\in\mathbb{N}\}\sup\{\beta_{n+1}:n\in\mathbb{N}\}  \\
   &\leq \sup\{\alpha_n:n\in\mathbb{N}\}\sup\{\beta_n:n\in\mathbb{N}\}<\infty.
\end{align*}
Therefore, the sequences $\{\lambda_n\gamma_n\}$ and $\{\alpha_n\beta_{n+1}\}$ are bounded, and the conclusion follows from Theorem \ref{isomorphism}.
\end{proof}

\begin{corollary}
 Let $\{((A_n, L_{n}),\phi_n)\}$ and $\{((A_n, L_{n}'),\phi_n)\}$ be two inductive sequences of compact quantum metric spaces. Suppose that for each $n\in\mathbb{N}$, there is a Lipschitz isomorphism $\psi_n: A_n\to A_n$ such that the diagram
  $$\CD
  A_1 @>\phi_1>> A_2 @>\phi_2>> A_3 @>\phi_3>>\cdots A\\
  @V V \psi_1V @V  V\psi_2V @V  V\psi_3V  \\
  A_1 @>\phi_1>> A_2@>\phi_2>> A_3 @>\phi_3>> \cdots A
\endCD
$$
is commutative, and such that
 \[
 L_{n}'(\psi_n(a))\leq \lambda_nL_{n}(a),\quad L_{n}(\psi_n^{-1}(a))\leq \alpha_nL_{n}'(a)
 \]
 for some $\alpha_n,\lambda_n>0$ and all $a\in A_n$. If the sequences $\{\lambda_n\}$ and $\{\alpha_n\}$ are bounded, and there are bounded sequences $\{\gamma_n\}$, $\{\beta_n\}$ and
  $\{\theta_n\}$  such that
  \begin{enumerate}
    \item $L'(\phi^n(a))\leq \gamma_nL_{n}'(a)$;
    \item $L(\phi^n(a))\leq \beta_nL_{n}(a)$;
    \item $L'_{n+1}(\phi_n(a))\leq \theta_nL_{n}'(a)$,
  \end{enumerate}
 for all $n\in\mathbb{N}$ and $a\in A_n$, then there is a Lipschitz isomorphism $\psi:A\to A$ such that
 \[
 L'(\psi(a))\leq 2\sup\{\lambda_n\gamma_n:n\in\mathbb{N}\}L(a)
 \]
 and
 \[
  L(\psi^{-1}(a))\leq 2\sup\{\theta_n\alpha_{n+1}\beta_{n+1}:n\in\mathbb{N}\} L'(a)
 \]
 for all $a\in A$.
\end{corollary}

\begin{proof}
 For any $n\in \mathbb{N}$, set
 \[
 \Psi_n=\psi_{n+1}^{-1}\circ \phi_n:A_n\to A_{n+1}.
 \]
 Then we get the following commutative diagram:
 \[\xymatrix{
  A_1\ar[d]^{\psi_1} \ar[r]^{\phi_1} & A_2 \ar[d]^{\psi_2} \ar[r]^{\phi_2} & A_3 \ar[d]^{\psi_3} \ar[r]^{\phi_3} & \cdots A \\
  A_1 \ar[r]^{\phi_1}\ar[ru]_{\Psi_1} & A_2 \ar[r]^{\phi_2} \ar[ru]_{\Psi_2}& A_3\ar[ru]_{\Psi_3} \ar[r]^{\phi_3} & \cdots A.}
  \]
  Moreover, for any $n\in \mathbb{N}$ and $a\in A_n$, we have
  \begin{align*}
    L_{n+1}(\Psi_n(a)) &=L_{n+1}((\psi_{n+1}^{-1}\circ \phi_n)(a))=L_{n+1}(\psi_{n+1}^{-1}(\phi_n(a)))  \\
     &\leq \alpha_{n+1}L_{n+1}'(\phi_n(a)) \\
     &\leq \theta_n\alpha_{n+1}L_{n}'(a).
  \end{align*}
  Note that
  \begin{align*}
 \sup\{\theta_n\alpha_{n+1}:n\in\mathbb{N}\}  &\leq \sup\{\theta_n:n\in\mathbb{N}\}\sup\{\alpha_{n+1}:n\in\mathbb{N}\}  \\
   &\leq \sup\{\theta_n:n\in\mathbb{N}\}\sup\{\alpha_n:n\in\mathbb{N}\}<\infty.
\end{align*}
 Therefore, the sequence $\{\theta_n\alpha_{n+1}\}$ is bounded, and the statement follows from Corollary \ref{isomorphism1}.
\end{proof}

%
%
%
%
%


%
%


\end{document}